\documentclass[12pt, reqno]{amsart}
\usepackage{amsmath, amsthm, amscd, amsfonts, amssymb, graphicx, color}
\usepackage[bookmarksnumbered, colorlinks, plainpages]{hyperref}

\newtheorem{theorem}{Theorem}[section]
\newtheorem{lemma}[theorem]{Lemma}

\newtheorem{corollary}[theorem]{Corollary}
\theoremstyle{definition}

\theoremstyle{remark}
\newtheorem{remark}[theorem]{Remark}
\numberwithin{equation}{section}

\begin{document}
\setcounter{page}{1}

\title[Operator Kantorovich and Wielandt inequalities]{Improved operator Kantorovich and Wielandt inequalities for positive linear maps}

\author[W.S. Liao, J.L. Wu]{Wenshi Liao$^1$$^{*}$ and Junliang Wu$^2$}

\address{$^{1}$ College of Mathematics and Statistics, Chongqing University, Chongqing, 401331, P.R. China.}
\email{\textcolor[rgb]{0.00,0.00,0.84}{liaowenshi@gmail.com}}

\address{$^{2}$ College of Mathematics and Statistics, Chongqing University, Chongqing, 401331, P.R. China.}
\email{\textcolor[rgb]{0.00,0.00,0.84}{jlwu678@163.com}}

%\dedicatory{This paper is dedicated to Professor ABCD}

\subjclass[2010]{Primary 47A63; Secondary 47A30.}

\keywords{Kantorovich inequality, Wielandt inequality, operator inequalities, positive linear maps.}

\date{Received: xxxxxx; Revised: yyyyyy; Accepted: zzzzzz.
\newline \indent $^{*}$ Corresponding author}

\begin{abstract}
In this paper, we improve and generalize the operator versions of Kantorovich 
and Wielandt inequalities for positive linear maps on Hilbert space.
Our results are more extensive and precise than many previous results due to 
Fu and He [Linear Multilinear Algebra, doi: 10. 1080/03081087. 2014. 880432.] 
and Zhang [Banach J. Math. Anal., 9 (2015), no. 1, 166-172.].
\end{abstract} \maketitle

\section{Introduction}

Throughout this paper, we reserve $M$, $m$ for real numbers and $I$ for the identity operator.
Other capital letters denote general elements of the $C^\ast $-algebra $\mathcal{B}(\mathcal{H})$
all bounded linear operators on a complex separable Hilbert space $(\langle \cdot,\cdot\rangle, \mathcal{H})$.
$\|\cdot\|$ denote the operator norm. An operator $A$ is said to be positive (strictly positive) if
$\langle Ax,x\rangle\ge0$ for all $x\in\mathcal{H}$ ($\langle Ax,x\rangle>0$ for all
 $x\in\mathcal{H}\backslash \{0\}$) and write $A\ge0$ ($A>0$).
$A\ge B$ ($A>B$) means $A-B\ge 0$ ($A-B>0$).
The absolute value of $A$ is denoted by $|A|$, that is, $|A|=(A^*A)^{\frac{1}{2}}$.

A linear map $\Phi:\mathcal{B}(\mathcal{H})\rightarrow \mathcal{B}(\mathcal{K})$ is called positive if
$\Phi(A)\ge 0$ whenever $A\ge 0$. It is said to be unital if $\Phi(I)=I$.
 We say that a linear map $\Phi$ between $C^*$-algebra is $2$-positive
if whenever the $2\times 2$ operator matrix
$\left[
  \begin{array}{cc}
    A  & B  \\
    B^* & C \\
  \end{array}
\right]\ge 0,$
then so is
$\left[
  \begin{array}{cc}
    \Phi(A)  & \Phi(B)  \\
    \Phi(B^*) & \Phi(C) \\
  \end{array}
\right]\ge 0.$

In 1948, Kantorovich \cite{Kantorovich} introduced the well-known Kantorovich inequality. In 1990,
 an operator Kantorovich inequality was established by Marshall and Olkin \cite{Marshall}. For recent
 development of the Kantorovich inequality, readers are referred to \cite{Moslehian}. Recently, Lin
 \cite{Lin2} proved that the operator Kantorovich inequality is order preserving under squaring. This result
 was further generalized by several authors (see \cite{Fu1, Zhang}), who obtained

 \begin{theorem}\cite{Fu1, Zhang}
 \label{thm1.1}
Let $0<m\le A\le M$. Then for every positive unital linear map $\Phi$,
\begin{equation}
\label{fu1}
\Phi^p(A^{-1})\le \frac{(m+M)^{2p}}{16m^pM^p}\Phi(A)^{-p},\quad p\ge2,
\end{equation}
\begin{equation}
\label{zhang1}
\Phi^p(A^{-1})\le \frac{(m^2+M^2)^{p}}{16m^pM^p}\Phi(A)^{-p},\quad p\ge4.
\end{equation}
\end{theorem}

When $p\ge 4$, the inequality \eqref{zhang1} is tighter than \eqref{fu1}.
There exsits a vacancy for $2\le p\le 4$. Motivated by this vacancy, we obtain some better results.

In view of $2$-positive linear map $\Phi$, Lin \cite{Lin2} proved that
\begin{equation}
\label{lin1}
|\Phi(A^{-1})\Phi(A)+\Phi(A)\Phi(A^{-1})|\le \frac{(M+m)^{2}}{2Mm},
\end{equation}
and
\begin{equation}
\label{lin2}
\Phi(A^{-1})\Phi(A)+\Phi(A)\Phi(A^{-1})\le \frac{(M+m)^{2}}{2Mm}.
\end{equation}

Fu \cite{Fu2} generalized the inequalities \eqref{lin1} and \eqref{lin2} to $p$-th power.

When considering the operator inequalities, we can not go without mentioning the operator means.
The axiomatic theory for operator means for pairs of positive operators have been developed by
Kubo and Ando \cite{Kubo}. A binary operation $\sigma$ defined on the set of strictly positive 
operators is called an operator mean provided that

(i) monotonity: $A\le C$ and $B\le D$ imply $A\sigma B\le C\sigma D$;

(ii) transformer inequality: $C^*(A\sigma B)C\le (C^*AC)\sigma (C^*BC)$ for every operator $C$;

(iii) upper continuity: $A_n\downarrow A$ and $B_n\downarrow B$ imply $A_n\sigma B_n\downarrow A\sigma B$,
 where $A_n\downarrow A$ means that $A_1\ge A_2\ge\cdots$ and
$A_n\rightarrow A$ as $n\rightarrow\infty$ in the strong operator topology;

(iv) normalization condition: $I\sigma I=I$.

As a matter of convenience, we use the following
notations to define the weighted arithmetic, geometric and harmonic means for operators:
\[
A\nabla _v B=(1-v)A+v B,~A\# _v B=A^{\frac{1}{2}}(A^{-\frac{1}{2}}BA^{-\frac{1}{2}})^v A^{\frac{1}{2}},~
\]
\[A! _v B=((1-v)A^{-1}+v B^{-1})^{-1},\]
where $A,B>0$ and $v \in [0,1]$. When $v =\frac{1}{2}$, we write  $A\nabla B$, $A\# B$ and $A! B$
for brevity, respectively.
The Young operator inequality proved in \cite{Furuta2} says that if $A,B>0$ and $v \in [0,1]$, then
\[
A\nabla _v B\ge A\# _v B.
\]
In terms of operator means, Hoa \cite{Hoa} obtained the following theorem:
\begin{theorem}\cite{Hoa}
Let $0<m\le A,B \le M$ and $\sigma$, $\tau$ be two arbitrary means
 between arithmetic and harmonic means. Then for every positive unital linear map $\Phi$,
\begin{equation}
\label{hoa1}
\Phi^2(A\sigma B)\le \mathrm{K}^2(h)\Phi^2(A\tau B),
\end{equation}
\begin{equation}
\label{hoa2}
\Phi^2(A\sigma B)\le \mathrm{K}^2(h)(\Phi(A)\tau\Phi(B))^2,
\end{equation}
\begin{equation}
\label{hoa3}
(\Phi(A)\sigma\Phi(B))^2\le \mathrm{K}^2(h)\Phi^2(A\tau B),
\end{equation}
and
\begin{equation}
\label{hoa4}
(\Phi(A)\sigma\Phi(B))^2\le \mathrm{K}^2(h)(\Phi(A)\tau\Phi(B))^2.
\end{equation}
where $\mathrm{K}(h)=\frac{(h+1)^2}{4h}$ with $h=\frac{M}{m}$ is the Kantorovich constant.
\end{theorem}

If we replace the exponent $2$ with $p$ ($0<p<2$), \eqref{hoa1}-\eqref{hoa4} are still true. In this paper,
we will study the case of $p>2$ via parameter $\alpha$.

\section{Kantorovich-type inequalities}

Firstly, we are devoted to obtain a better bound than \eqref{fu1} and \eqref{zhang1}. 
In order to do that, we need two important lemmas.

\begin{lemma}\cite[Lemma 2.1]{Bhatia2}
\label{lema1}
Let $A, B\ge 0$. Then the following inequality holds:
\begin{equation}
\label{lem1}
\|AB\|\le\frac{1}{4}\|A+B\|^2.
\end{equation}
\end{lemma}
\label{lema2}
\begin{lemma}\cite[p. 28]{Bhatia}
Let $A, B\ge 0$. Then for $1\le r<+\infty$,
\begin{equation}
\label{lem2}
\|A^r+B^r\|\le\|(A+B)^r\|.
\end{equation}
\end{lemma}

We know that $\|A\|\le 1$ is equivalent to $A^*A\le I$. Using this fact we have the following excellent theorem:
\begin{theorem}
Let $0<m\le A\le M$. Then for every positive unital linear map $\Phi$, $1\le\alpha\le 2$ and $p\ge2\alpha$,
\begin{equation}
\label{th21}
\Phi^p(A^{-1})\le \frac{(m^\alpha+M^\alpha)^{\frac{2p}{\alpha}}}{16m^pM^p}\Phi(A)^{-p}.
\end{equation}
\end{theorem}

\begin{proof}
The desired inequality is equivalent to
\[\|\Phi^{\frac{p}{2}}(A^{-1})\Phi^{\frac{p}{2}}(A)\|\le \frac{(m^\alpha+M^\alpha)^{\frac{p}{\alpha}}}{4m^{\frac{p}{2}}M^{\frac{p}{2}}}.\]
Compute
\begin{align*}
\|m^{\frac{p}{2}}M^{\frac{p}{2}}\Phi^{\frac{p}{2}}(A^{-1})\Phi^{\frac{p}{2}}(A)\|
&\le\frac{1}{4}\|m^{\frac{p}{2}}M^{\frac{p}{2}}\Phi^{\frac{p}{2}}(A^{-1})+\Phi^{\frac{p}{2}}(A)\|^2\quad (\mathrm{by}~(2.1))\\
&\le \frac{1}{4}\|\left(m^{\alpha}M^{\alpha}\Phi^{\alpha}(A^{-1})+\Phi^{\alpha}(A)\right)^{\frac{p}{2\alpha}}\|^2 \quad (\mathrm{by}~(2.2))\\
&=\frac{1}{4}\|m^{\alpha}M^{\alpha}\Phi^{\alpha}(A^{-1})+\Phi^{\alpha}(A)\|^{\frac{p}{\alpha}}\\
&\le\frac{1}{4}(m^{\alpha}+M^{\alpha})^{\frac{p}{\alpha}}.
\end{align*}
The last inequality above holds as follows:
The condition $0<m\le A\le M$ implies that
\begin{equation}
\label{th21a}
M^{\alpha}m^{\alpha}A^{-\alpha}+A^{\alpha}\le M^{\alpha}+m^{\alpha},
\end{equation}
and hence
\begin{equation}
\label{th211}
M^{\alpha}m^{\alpha}\Phi(A^{-\alpha})+\Phi(A^{\alpha})\le M^{\alpha}+m^{\alpha},
\end{equation}
The well-known inequality (see \cite[p. 53]{Bhatia}) says
\[\Phi^{\alpha}(T)\le\Phi(T^{\alpha})\] for every positive unital linear map $\Phi$ and $T>0$.
Then it follows from \eqref{th211} that
\[
M^{\alpha}m^{\alpha}\Phi^{\alpha}(A^{-1})+\Phi^{\alpha}(A)\le M^{\alpha}+m^{\alpha}.
\]
Therefore
\begin{equation}
\label{th213}
\|\Phi^{\frac{p}{2}}(A^{-1})\Phi^{\frac{p}{2}}(A)\|\le \frac{(m^\alpha+M^\alpha)^{\frac{p}{\alpha}}}{4m^{\frac{p}{2}}M^{\frac{p}{2}}}.
\end{equation}
So the inequality \eqref{th21} has been obtained.
\end{proof}

\begin{remark}\label{rem24}
Inequalities \eqref{fu1} and \eqref{zhang1} are two special cases of Theorem \ref{th21} by taking $\alpha=1,2$.

Put $\alpha=2$ and $p=4$, the inequlaity \eqref{th211} reduces to Lin's result (see \cite[Theorem 4.3]{Lin1}).

Our argument depends essentially on a result of Hardy, Little-wood and P\'{o}lya (see \cite{Hardy}, p. 28) that the function $f(\alpha)=\left(M^{\alpha}+m^{\alpha}\right)^{\frac{1}{\alpha}}$ ($M, m>0$  and  $\alpha> 0$) is monotone decreasing,
so we can conclude that Theorem \ref{th21} is sharper than \eqref{fu1} as $p>2$.
\end{remark}

We next present the generalizations of \eqref{lin1} and \eqref{lin2}. The following 
lemma is useful in our derivative of Theorem \ref{thm22}.
\begin{lemma}\label{lem4}
For any bounded operator $X$,
\[|X|\le tI\Leftrightarrow \|X\|\le t\Leftrightarrow \left[
  \begin{array}{cc}
    tI  & X  \\
    X^* & tI \\
  \end{array}
\right]\ge 0  ~(t\ge0).
\]
\end{lemma}

\begin{theorem}\label{thm22}
Let $A$ be a positive operator on a Hilbert space $\mathcal{H}$ with $0<m\le A\le M$ and $\Phi$
be a $2$-positive linear map on $\mathcal{B(H)}$. Then for $1\le\alpha\le 2$ and $p\ge \alpha$,
\begin{equation}
\label{th22}
|\Phi^p(A^{-1})\Phi^p(A)+\Phi^p(A)\Phi^p(A^{-1})|\le \frac{(M^\alpha+m^\alpha)^{\frac{2p}{\alpha}}}{2M^pm^p}
\end{equation}
and
\begin{equation}
\label{th221}
\Phi^p(A^{-1})\Phi^p(A)+\Phi^p(A)\Phi^p(A^{-1})\le \frac{(M^\alpha+m^\alpha)^{\frac{2p}{\alpha}}}{2M^pm^p}.
\end{equation}
\end{theorem}

\begin{proof}
By \eqref{th213} and Lemma \ref{lem4}, we deduce
\[\left[
  \begin{array}{cc}
   \frac{(M^\alpha+m^\alpha)^{\frac{2p}{\alpha}}}{4M^pm^p}I  & \Phi^p(A)\Phi^p(A^{-1}) \\
    \Phi^p(A^{-1})\Phi^p(A) & \frac{(M^\alpha+m^\alpha)^{\frac{2p}{\alpha}}}{4M^pm^p}I \\
  \end{array}
\right]\ge 0
\]
and
\[\left[
  \begin{array}{cc}
   \frac{(M^\alpha+m^\alpha)^{\frac{2p}{\alpha}}}{4M^pm^p}I  & \Phi^p(A^{-1})\Phi^p(A) \\
     \Phi^p(A)\Phi^p(A^{-1}) & \frac{(M^\alpha+m^\alpha)^{\frac{2p}{\alpha}}}{4M^pm^p}I \\
  \end{array}
\right]\ge 0.
\]
Summing up these two operator matrices, we have
\[\left[
  \begin{array}{cc}
   \frac{(M^\alpha+m^\alpha)^{\frac{2p}{\alpha}}}{2M^pm^p}I  & \Phi^p(A)\Phi^p(A^{-1})+\Phi^p(A^{-1})\Phi^p(A) \\
     \Phi^p(A^{-1})\Phi^p(A)+\Phi^p(A)\Phi^p(A^{-1}) & \frac{(M^\alpha+m^\alpha)^{\frac{2p}{\alpha}}}{2M^pm^p}I \\
  \end{array}
\right]\ge 0.
\]
By Lemma \ref{lem4} again, we obtain \eqref{th22}.

As $\Phi^p(A)\Phi^p(A^{-1})+\Phi^p(A^{-1})\Phi^p(A)$ is self-adjoint, \eqref{th221} follows from
 the maximal characterization of geometric mean.
\end{proof}

\begin{remark}
Take $\alpha=1$ and $p=1$, \eqref{th22} and \eqref{th221} collapse to \eqref{lin1} and \eqref{lin2}, respectively.

Fu showed a special case of Theorem \ref{thm22} as $\alpha=1$ in \cite[Theorem 4]{Fu2}.
\end{remark}

When $\alpha=2$, Theorem \ref{thm22} imply the following.
\begin{corollary}
Let $A$ be a positive operator on a Hilbert space $\mathcal{H}$ with $0<m\le A\le M$ and $\Phi$
be a $2$-positive linear map on $\mathcal{B(H)}$. Then for $p\ge 2$ ,
\begin{equation}
\label{cor11}
|\Phi^p(A^{-1})\Phi^p(A)+\Phi^p(A)\Phi^p(A^{-1})|\le \frac{(M^2+m^2)^{p}}{2M^pm^p},
\end{equation}
and
\begin{equation}
\label{cor12}
\Phi^p(A^{-1})\Phi^p(A)+\Phi^p(A)\Phi^p(A^{-1})\le \frac{(M^2+m^2)^{p}}{2M^pm^p}.
\end{equation}
\end{corollary}
\begin{remark}
When $p\ge 2$, the inequalities \eqref{cor11} and \eqref{cor12} is tighter than 
that of Fu \cite[Theorem 4]{Fu2}, respectively.
\end{remark}

Using the similar idea of Zhang \cite{Zhang}, we deduce the following theorem about operator means.
Firstly, we point out that there is a gap in the proof of Theorem 2.6 (see \cite{Zhang}). 
Author get the last inequality
\begin{equation}
\label{zhanga}
\frac{1}{4}\|k\Phi^{2}(\frac{A+B}{2})+km^2 M^2\Phi^{-2}(\frac{A+B}{2})\|^{\frac{p}{2}}
\le\frac{1}{4}(k(m^2+M^2))^{\frac{p}{2}}
\end{equation}
by \[M^{2}m^{2}\Phi^{2}(A^{-1})+\Phi^{2}(A)\le M^{2}+m^{2},\]
see \cite[(4.7)]{Lin1}. However, by the inequality above, we can not get
\[M^{2}m^{2}\Phi^{-2}(A)+\Phi^{2}(A)\le M^{2}+m^{2}.\]
Actually, using the inequality
\[M^{2}m^{2}A^{-2}+A^{2}\le M^{2}+m^{2},\]
see \cite{Lin1}, we can obtain \eqref{zhanga} directly.
\begin{theorem}
Let $\sigma$ and $\tau$ be two arbitrary means between harmonic and arithmetic mean. 
Let $0<m\le A, B\le M$, $1<\alpha\le 2$ and $p\ge2\alpha$. Then for every positive unital linear map $\Phi$,
\begin{equation}
\label{th231}
\Phi^p(A\sigma B)\le \frac{(k^{\frac{\alpha}{2}}(M^\alpha+m^\alpha))^{\frac{2p}{\alpha}}}{16M^pm^p}\Phi^p(A\tau B),
\end{equation}
\begin{equation}
\label{th232}
\Phi^p(A\sigma B)\le \frac{(k^{\frac{\alpha}{2}}(M^\alpha+m^\alpha))^{\frac{2p}{\alpha}}}{16M^pm^p}(\Phi(A)\tau\Phi(B))^p,
\end{equation}
\begin{equation}
\label{th233}
(\Phi(A)\sigma \Phi(B))^p\le \frac{(k^{\frac{\alpha}{2}}(M^\alpha+m^\alpha))^{\frac{2p}{\alpha}}}{16M^pm^p}\Phi^p(A\tau B),
\end{equation}
and
\begin{equation}
\label{th234}
(\Phi(A)\sigma \Phi(B))^p\le \frac{(k^{\frac{\alpha}{2}}(M^\alpha+m^\alpha))^{\frac{2p}{\alpha}}}{16M^pm^p}(\Phi(A)\tau\Phi(B))^p,
\end{equation}
where $k=\mathrm{K}(h)=\frac{(h+1)^2}{4h}$ with $h=\frac{M}{m}$ is the Kantorovich constant.
\end{theorem}

\begin{proof}
It follows from the inequality \eqref{hoa1} that
\begin{equation}
\label{th235}
k^\alpha\Phi^{-\alpha}(A\sigma B)\ge\Phi^{-\alpha}(A\tau B).
\end{equation}
 As we know, the relation \eqref{th231} is equivalent to
 \[
 \Phi^{-\frac{p}{2}}(A\tau B)\Phi^p(A\sigma B)\Phi^{-\frac{p}{2}}(A\tau B)\le 
 \frac{(k^{\frac{\alpha}{2}}(M^\alpha+m^\alpha))^{\frac{2p}{\alpha}}}{16M^pm^p},
\]
or
\[
 \|\Phi^{\frac{p}{2}}(A\sigma B)\Phi^{-\frac{p}{2}}(A\tau B)\|\le 
 \frac{(k^{\frac{\alpha}{2}}(M^\alpha+m^\alpha))^{\frac{p}{\alpha}}}{4M^{\frac{p}{2}}m^{\frac{p}{2}}}.
\]
Combing Lemma \ref{lem1}, Lemma \ref{lem2} with \eqref{th235}, we get
\begin{align*}
\|m^{\frac{p}{2}}M^{\frac{p}{2}}&\Phi^{\frac{p}{2}}(A\sigma B)\Phi^{-\frac{p}{2}}(A\tau B)\|\\
&\le \frac{1}{4}\|k^{\frac{p}{4}}\Phi^{\frac{p}{2}}(A\sigma B)+k^{-\frac{p}{4}}m^{\frac{p}{2}} M^{\frac{p}{2}}\Phi^{-\frac{p}{2}}(A\tau B)\|^2\\
&\le \frac{1}{4}\|\left(k^{\frac{\alpha}{2}}\Phi^{\alpha}(A\sigma B)+k^{-\frac{\alpha}{2}}m^\alpha M^\alpha\Phi^{-\alpha}(A\tau B)\right)^{\frac{p}{2\alpha}}\|^2\\
&=\frac{1}{4}\|k^{\frac{\alpha}{2}}\Phi^{\alpha}(A\sigma B)+k^{-\frac{\alpha}{2}}m^\alpha M^\alpha\Phi^{-\alpha}(A\tau B)\|^{\frac{p}{\alpha}}\\
&\le\frac{1}{4}\|k^{\frac{\alpha}{2}}\Phi^{\alpha}(A\sigma B)+k^{\frac{\alpha}{2}}m^\alpha M^\alpha\Phi^{-\alpha}(A\sigma B)\|^{\frac{p}{\alpha}}\\
&\le\frac{1}{4}k^{\frac{p}{2}}(m^\alpha+M^\alpha)^{\frac{p}{\alpha}}.
\end{align*}
The last inequality follows from \eqref{th21a} and $0<m\le A\sigma B\le M$ that
\[
M^{\alpha}m^{\alpha}\Phi^{-\alpha}(A\sigma B)+\Phi^{\alpha}(A\sigma B)\le M^{\alpha}+m^{\alpha}.
\]
The inequality \eqref{th231} is completed. Similarly, \eqref{th232}-\eqref{th234} 
can be derived from the inequalities \eqref{hoa2}-\eqref{hoa4}, respectively.
\end{proof}

\section{Wielandt-type inequalities}

 In 2000, Bhatia and Davis \cite{Bhatia1} proved an operator Wielandt inequality which states that
if $0<m\le A\le M$ and $X$, $Y$ are two partial isometries on $\mathcal{H}$ whose final spaces are orthogonal to
each other, then for every $2$-positive linear map $\Phi$ on $\mathcal{B(H)}$,
\[
\Phi(X^*AY)\Phi(Y^*AY)^{-1}\Phi(Y^*AX)\le \left(\frac{M-m}{M+m}\right)^2\Phi(X^*AX).
\]
Lin \cite[Conjecture 3.4]{Lin2} conjecture that the following assertion could be true:
\begin{equation}
\label{lin3}
\|\Phi(X^*AY)\Phi(Y^*AY)^{-1}\Phi(Y^*AX)\Phi(X^*AX)^{-1}\|\le \left(\frac{M-m}{M+m}\right)^2.
\end{equation}

Recently, Fu and He \cite{Fu1} attempt to solve the conjecture and get a step closer to the conjecture.
But Gumus \cite {Gumus} obtain a better upper bound to approximate the right side of \eqref{lin3} based on
\begin{equation}
\label{gumus1}
\|\Phi(X^*AY)\Phi(Y^*AY)^{-1}\Phi(Y^*AX)\Phi(X^*AX)^{-1}\|\le \frac{(M-m)^2}{2(M+m)\sqrt{Mm}},
\end{equation}
which is equivalent to
\begin{equation}
\label{gumus2}
\left(\Phi(X^*AY)\Phi(Y^*AY)^{-1}\Phi(Y^*AX)\right)^2\le \frac{(M-m)^4}{4(M+m)^2Mm}\Phi^2(X^*AX).
\end{equation}

Soon after, Zhang \cite{Zhang} proved a generalized inequality
\begin{equation}
\label{zhang}
\begin{split}
&\|\left(\Phi(X^*AY)\Phi(Y^*AY)^{-1}\Phi(Y^*AX)\right)^{\frac{p}{2}}\Phi(X^*AX)^{-\frac{p}{2}}\|\\
&\le \frac{1}{4}\left(\left(\frac{M-m}{M+m}\right)^2M+\frac{1}{m} \right)^p, \quad p\ge 2.
\end{split}
\end{equation}
\begin{equation}
\label{zhang3}
\begin{split}
&\|\left(\Phi(X^*AY)\Phi(Y^*AY)^{-1}\Phi(Y^*AX)\right)^{\frac{p}{2}}\Phi(X^*AX)^{-\frac{p}{2}}\|\\
&\le \left(\frac{M-m}{M+m}\right)^p\left(\frac{M}{m}\right)^{\frac{p}{2}}, \quad p\ge 1.
\end{split}
\end{equation}
Now, Let us give an improvement of \eqref{zhang}.
\begin{theorem}\label{thm31}
Let $0<m\le A\le M$, $X$ and $Y$ be two isometries in $\mathcal{H}$ whose final spaces are orthogonal to
each other and $\Phi$ be a 2-positive linear map on $\mathcal{B(H)}$. Then for $1\le\alpha\le 2$ and $p\ge 2\alpha$,
\begin{equation}
\label{th31}
\begin{split}
&\|\left(\Phi(X^*AY)\Phi(Y^*AY)^{-1}\Phi(Y^*AX)\right)^{\frac{p}{2}}\Phi(X^*AX)^{-\frac{p}{2}}\|\\
&\le \frac{(M-m)^{p}(M^\alpha+m^\alpha)^{\frac{p}{\alpha}}}{2^{2+\frac{p}{2}} M^{\frac{3p}{4}}m^{\frac{3p}{4}}(M+m)^{\frac{p}{2}}}.
\end{split}
\end{equation}
\end{theorem}

\begin{proof}
By using \eqref{gumus2}, we have
\begin{equation}
\label{th311}
\left(\Phi(X^*AY)\Phi(Y^*AY)^{-1}\Phi(Y^*AX)\right)^\alpha\le 
\frac{(M-m)^{2\alpha}}{2^\alpha(M+m)^\alpha M^{\frac{\alpha}{2}}m^{\frac{\alpha}{2}}}\Phi^\alpha(X^*AX).
\end{equation}
Combing Lemma \ref{lem1}, Lemma \ref{lem2} with \eqref{th311}, we get

\begin{align*}
&\left\|\left(\Phi(X^*AY)\Phi(Y^*AY)^{-1}\Phi(Y^*AX)\right)^{\frac{p}{2}}\frac{(M-m)^{p}}{2^{\frac{p}{2}}(M+m)^{\frac{p}{2}}}M^{\frac{p}{4}}m^{\frac{p}{4}}\Phi(X^*AX)^{-\frac{p}{2}}\right\|\\
&\le \frac{1}{4}\left\|\left(\Phi(X^*AY)\Phi(Y^*AY)^{-1}\Phi(Y^*AX)\right)^{\frac{p}{2}}+\left(\frac{(M-m)^2}{2(M+m)}\sqrt{Mm}\Phi(X^*AX)^{-1}\right)^{\frac{p}{2}}\right\|^2\\
&\le \frac{1}{4}\left\|\left(\Phi(X^*AY)\Phi(Y^*AY)^{-1}\Phi(Y^*AX)\right)^\alpha+\frac{(M-m)^{2\alpha}}{2^\alpha(M+m)^\alpha}(M m)^{\frac{\alpha}{2}}\Phi(X^*AX)^{-\alpha}\right\|^{\frac{p}{\alpha}}\\
&\le \frac{1}{4}\left\|\frac{(M-m)^{2\alpha}}{2^\alpha (Mm)^{\frac{\alpha}{2}}(M+m)^{\alpha}}\Phi(X^*AX)^{\alpha}+\frac{(M-m)^{2\alpha}}{2^\alpha(M+m)^\alpha}(M m)^{\frac{\alpha}{2}}\Phi(X^*AX)^{-\alpha}\right\|^{\frac{p}{\alpha}}\\
&=\frac{(M-m)^{2p}}{2^{2+p} M^{\frac{p}{2}}m^{\frac{p}{2}}(M+m)^p}\left\|\Phi(X^*AX)^{\alpha}+M^\alpha m^\alpha\Phi(X^*AX)^{-\alpha}\right\|^{\frac{p}{\alpha}}\\
&\le\frac{(M-m)^{2p}(M^\alpha+m^\alpha)^{\frac{p}{\alpha}}}{2^{2+p} M^{\frac{p}{2}}m^{\frac{p}{2}}(M+m)^p}.
\end{align*}
The last inequality follows from \eqref{th21a} and $0<m\le X^*AX\le M$. This proves the inequality \eqref{th31}.
\end{proof}

Putting $\alpha=1,2$ in Theorem \ref{thm31}, we have
\begin{corollary}
Under the same conditions as in Theorem \ref{thm31}, then
\begin{equation}
\label{cor31}
\begin{split}
&\|\left(\Phi(X^*AY)\Phi(Y^*AY)^{-1}\Phi(Y^*AX)\right)^{\frac{p}{2}}\Phi(X^*AX)^{-\frac{p}{2}}\|\\
&\le \frac{(M-m)^{p}(M+m)^{\frac{p}{2}}}{2^{2+\frac{p}{2}} M^{\frac{3p}{4}}m^{\frac{3p}{4}}},~~~p\ge 2
\end{split}
\end{equation}
and
\begin{equation}
\label{cor32}
\begin{split}
&\|\left(\Phi(X^*AY)\Phi(Y^*AY)^{-1}\Phi(Y^*AX)\right)^{\frac{p}{2}}\Phi(X^*AX)^{-\frac{p}{2}}\|\\
&\le \frac{(M-m)^{p}(M^2+m^2)^{\frac{p}{2}}}{2^{2+\frac{p}{2}} M^{\frac{3p}{4}}m^{\frac{3p}{4}}(M+m)^{\frac{p}{2}}},~~~p\ge 4.
\end{split}
\end{equation}
\end{corollary}

\begin{remark}
\eqref{cor32} is better than \eqref{cor31} as $p\ge 4$.
A simple computation shows that
\[\frac{(M-m)^2(M+m)}{8M^{\frac{3}{2}}m^{\frac{3}{2}}}\le 
\left(\frac{M-m}{M+m}\right)^2\frac{M}{m}\le \frac{1}{4}\left(\left(\frac{M-m}{M+m}\right)^2M+\frac{1}{m} \right)^2,\]
which points out that the right side of \eqref{cor31} is a better bound than that of \eqref{zhang} and \eqref{zhang3} as $p\ge 2$.

Similar to that of Remark \ref{rem24}, by using the function $f(\alpha)=\left(M^{\alpha}+m^{\alpha}\right)^{\frac{1}{\alpha}}$,
we derive that the upper bound in \eqref{th31} is tighter than that in \eqref{cor31} as $p> 2$.
\end{remark}
Following from \eqref{th31} and the line of the prove of Theorem \ref{thm22}, we have
\begin{theorem}
Under the same conditions as in Theorem \ref{thm31}, denote
\[\Gamma=\left(\Phi(X^*AY)\Phi(Y^*AY)^{-1}\Phi(Y^*AX)\right)^{\frac{p}{2}}\Phi(X^*AX)^{-\frac{p}{2}},\] then
\[
|\Gamma+\Gamma^*|\le \frac{(M-m)^{p}(M^\alpha+m^\alpha)^{\frac{p}{\alpha}}}{2^{1+\frac{p}{2}} M^{\frac{3p}{4}}m^{\frac{3p}{4}}(M+m)^{\frac{p}{2}}}
\]
and
\[
\Gamma+\Gamma^*\le \frac{(M-m)^{p}(M^\alpha+m^\alpha)^{\frac{p}{\alpha}}}{2^{1+\frac{p}{2}} M^{\frac{3p}{4}}m^{\frac{3p}{4}}(M+m)^{\frac{p}{2}}}.
\]
\end{theorem}

\section{Suppliments}
Ando's inequality \cite{Ando} asserts that if $A,B\in \mathcal{B}(\mathcal{H})$ are positive operators, $v\in[0,1]$ and
$\Phi$ is a positive linear map, then
\[
  \Phi(A\sharp_vB)\le \Phi(A)\sharp_v\Phi(B).
\]

\begin{theorem}\label{th41}
Let $0<mI\le A, B\le MI$. Then for every positive unital linear map $\Phi$ and $0<v<1$,
\begin{equation}
\label{th412}
\Phi^2(A\nabla_v B)\le \mathrm{K}^2(h)\Phi^2(A\sharp_vB)
\end{equation}
and
\begin{equation}
\label{th413}
\Phi^2(A\nabla_v B)\le \mathrm{K}^2(h)(\Phi(A)\sharp_v\Phi(B))^2,
\end{equation}
where $\mathrm{K}(h)=\frac{(h+1)^2}{4h}$ with $h=\frac{M}{m}$.
\end{theorem}

\begin{proof}
The inequality \eqref{th412} is equivalent to
\begin{equation}
\label{th414}
\|\Phi(A\nabla_v B)\Phi^{-1}(A\sharp_vB)\|\le \mathrm{K}(h).
\end{equation}
By Lemma 2.2, \eqref{th414} is valid if
\begin{equation}
\label{th415}
\Phi(A\nabla_v B)+Mm\Phi^{-1}(A\sharp_vB)\le M+m.
\end{equation}
The well-known Choi inequality (see \cite [p. 41]{Bhatia}) says that
\[\Phi(T^{-1})\ge \Phi^{-1}(T)\quad for~any~T>0,\]
so \eqref{th415} can be derived if
\begin{equation}
\label{th416}
\Phi(A\nabla_v B)+Mm\Phi((A\sharp_vB)^{-1})\le M+m.
\end{equation}
Indeed, we shall show a stronger inequality than \eqref{th416}, it is referred to Lin (see \cite[(2.3)]{Lin2})
\[\Phi(A)+Mm\Phi(A^{-1})\le M+m,\]
\[\Phi(B)+Mm\Phi(B^{-1})\le M+m,\]
summing up the two inequalities above by multiplying $1-v$ and $v$, respectively, we get
\begin{equation}
\label{th417}
\Phi((1-v)A+vB)+Mm\Phi((1-v)A^{-1}+vB^{-1})\le M+m.
\end{equation}
Note that \eqref{th417} is tighter than \eqref{th416}, since $(A!_vB)^{-1}\ge A^{-1}\sharp_vB^{-1} =(A\sharp_vB)^{-1}$ and $\Phi$ is order preserving.
\eqref{th412} is completed. The proof of \eqref{th413} is similar.
\end{proof}

Note that
$\Phi(A\sharp_vB)\le \Phi(A)\sharp_v\Phi(B)$ can not be squared (see, \cite[Proposition 1.2]{Lin1}),
 so we can not say that \eqref{th412} is stronger than \eqref{th413}.
\begin{remark}
In the proof of Theorem \ref{th41}, we can conclude that
\[
\Phi^2(A\nabla_v B)\le \mathrm{K}^2(h)\Phi^2(A!_vB),
\]
and
\[
\Phi^2(A\nabla_v B)\le \mathrm{K}^2(h)(\Phi(A)!_v\Phi(B))^2,
\]
which show that the direct relations of arithmetic mean and harmonic mean.
\end{remark}
Several reverse Young operator inequalities appeared in \cite{Liao, Tominaga} as follows:
\begin{theorem}\cite{Liao, Tominaga}
Let $A,B$ be two invertible positive operators and
satisfy $0<m\le A, B\le M$.
Then
\[
A\nabla _v B \le {\rm K}(h)^R A\# _v B
\]
and
\[
 A\nabla _v B \le S(h) A\# _v B
\]
for all $v \in [0, 1]$, where $R = \max \{1 - v ,v \}$, $h=\frac{M}{m}$ and 
$S(h)= \dfrac{h^{\frac{1}{h - 1}}}{e\ln h^{\frac{1}{h - 1}}}(h\ne
1)$ is the Specht's ratio.
\end{theorem}

It is natural to ask whether the following analogue of \eqref{th412} holds or not. Under
the same condition of Theorem \ref{th41}, are they true
\[
\Phi^2(A\nabla_v B)\le \mathrm{K}^{2R}(h)\Phi^2(A\sharp_vB)
\]
and
\[
\Phi^2(A\nabla_v B)\le S^2(h)\Phi^2(A\sharp_vB)?
\]
 We have been unable to prove (or disprove) them.

%{\bf Acknowledgement.} Acknowledgements could be placed at the end
%of the text but precede the references.

\bibliographystyle{amsplain}

\end{document}